\documentclass{amsart}
\usepackage{amssymb}
\usepackage{xypic}
\newtheorem{theorem}{Theorem}
\newtheorem{definition}[theorem]{Definition}
\newcommand{\bd}[1]{\begin{definition}\label{d-#1}\rm}
\newcommand{\ed}{\end{definition}}
\newtheorem{lemma}[theorem]{Lemma}
\newtheorem{cor}[theorem]{Corollary}
\newcommand{\bc}[1]{\begin{cor}\label{c-#1}}
\newcommand{\ec}{\end{cor}}
\newcommand{\bl}[1]{\begin{lemma}\label{l-#1}} 
\newcommand{\el}{\end{lemma}}               
\newcommand{\HNN}{\text{HNN}}
\newcommand{\gp}[1]{\langle#1\rangle}
\newcommand{\ugp}{\{1\}}
\newcommand{\prefdef}[1]{Definition \ref{d-#1}}
\newcommand{\prefcor}[1]{Corollary \ref{c-#1}}
\newcommand{\tG}{\tilde G}
\newcommand{\inv}{^{-1}}
\title[Virtually Free pro-$p$ groups whose Torsion]%
 {Addendum: Virtually Free pro-$p$ groups whose Torsion Elements have finite Centralizer}%

\author{W. Herfort and P.A. Zalesskii}

\newcommand{\classno}{20E18 (primary), 20E06, 22C05 (secondary).}

\begin{document}

\begin{abstract}
We fill details in the proof of \cite[Lemma 13]{hz} (that is \cite[Lemma 3.2]{hza}). 
For easier reading we include
the relevant part of section 3 ibidem.
\medskip
 
\classno
\end{abstract}
\maketitle
\setcounter{section}2
\section{HNN-embedding}
We introduce a notion of a pro-$p$ \HNN-group as a
generalization of pro-$p$ \HNN-extension in the sense of
\cite[page 97]{RZ2}. It also can be defined as a sequence of pro-$p$
\HNN-extensions. During the definition to follow, $i$ belongs to a
finite set $I$ of indices.

\bd{HNN-grp} Let $G$ be a pro-$p$ group and $A_i, B_i$ be
subgroups of $G$ with isomorphisms $\phi_i:A_i\longrightarrow
B_i$. The pro-$p$ \HNN-group  is then a pro-$p$ group
$\HNN(G,A_i,\phi_i,z_i)$ having presentation
$$\HNN(G,A_i,\phi_i,z_i)=\langle G, z_i\mid rel(G), \forall a_i\in A_i:\ \
a_i^{z_i}=\phi_i(a_i)\rangle.$$ The group $G$ is called the {\em
base group}, $A_i,B_i$ are called {\em associated subgroups} and
$z_i$ are called the {\em stable letters}. \ed

For the rest of this section let $G$ be a finitely generated virtually free pro-$p$ group, and
fix an open free pro-$p$ normal subgroup $F$ of $G$ of minimal index.
Also suppose that $C_F(t)=\ugp$ for every torsion element $t\in
G$. Let $K:=G/F$ and form $G_0:=G\amalg K$. Let
$\psi:G\to K$ denote the canonical projection.
It extends to an epimorphism $\psi_0:G_0\to K$, 
by sending $g\in G$ to $gF/F\in K$ and each $k\in K$ identically to $k$, and using the universal
property of the free pro-$p$ product. Remark that the kernel
of $\psi_0$, say $L$, 
is an open subgroup of $G_0$ and, as $L\cap G=F$ and $L\cap K=\ugp$,
as a consequence of the pro-$p$ version of the Kurosh subgroup theorem, \cite[Theorem 9.1.9]{RZ1 2000},
$L$ is free pro-$p$.
Let $I$ be the set
of all $G$-conjugacy classes of maximal finite subgroups of $G$ and
observe that in light of \cite[Lemma 8]{hz} the set $I$ is finite.
Fix, for every $i\in I$, a finite subgroup $K_i$ of
$G$ in the $G$-conjugacy class $i$. We define a pro-$p$ \HNN-group
by considering first  $\tG_0:=G_0\amalg F(z_i\mid i\in I)$ with $z_i$
constituting a free set of generators, and then taking the normal
subgroup $R$ in $\tG_0$ generated by all elements of the form
$k_i^{z_i}\psi(k_i)\inv$, with $k_i\in K_i$ and $i\in I$. Finally
set
$$\tG:=\tG_0/R,$$
and, since all $K_i$ are finite, by \cite[Prop.\,9.4.3]{RZ1 2000} it is a proper \HNN-group
$\HNN(G_0, K_i,\phi_i,z_i)$,
where $\phi_i:=\psi_{|K_i}$,  $G_0$ is the base group, the $K_i$
are associated subgroups, and the $z_i$ form a set of stable
letters in the sense of \prefdef{HNN-grp}.

\medskip
Let us show that $\tG$ is virtually free pro-$p$. The above epimorphism 
$\psi_0:G_0\longrightarrow K$ extends to $\tG\longrightarrow K$
by the universal property of the HNN-extension, so $\tG$ is a
semidirect product $\tilde F\rtimes K$ of its kernel $\tilde F$
with $K$. By \cite[Lemma 10]{HZ 07}, every open torsion-free
subgroup of $\tG$ must be free pro-$p$, so $\tilde F$ is free pro-$p$.

\medskip

\newcommand{\N}{\mathbb N}

The objective of the section is to give a more detailed version of the proof of \cite[Lemma 13]{hz}, i.e.,
to  show that the centralizers of torsion elements in $\tilde G$ are finite.

\bl{centr} Let $\tilde G= \HNN(G_0,K_i,\phi_i,z_i)$ and $\tilde F$
be as explained.   Then $C_{\tilde F}(t)=1$ for every torsion
element $t\in \tilde G$. \el

\begin{proof} There is a standard pro-$p$ tree $S:=S(\tilde G)$
associated to $\tilde G:= \HNN(G_0,K_i,\phi_i,z_i)$ on which
$\tilde G$ acts naturally such that the vertex stabilizers are
conjugates of $G_0$ and each edge stabilizer is a conjugate of
some $K_i$.

\medskip

\noindent{\em Claim:
Let $e_1,e_2$ be two edges of $S$ with a common vertex
$v$ which is not terminal vertex of both of them. Then the intersection of the stabilizers $\tilde
G_{e_1}\cap\tilde G_{e_2}$ is trivial.
}

\medskip

{\em Proof of the Claim: } By translating $e_1$, $e_2$, $v$ if necessary we may assume
that $G_0$ is the stabilizer of $v$. Then  we
have two cases:

1) $v$ is initial vertex of $e_1$ and $e_2$. Then $\tilde G_{e_1}=K_i^g$ and 
$\tilde G_{e_2}=K_{i'}^{g'}$ with $g,g'\in G_0$
and either $i\neq i'$ or $g\not\in K_ig'$ and by construction of $\tilde G$ 
we have $K_i^g\neq K_{i'}^{g'}$ if $t\neq t'$.
Suppose that $K_i^g\cap K_{i'}^{g'}\neq\ugp$. Then, since $G_0=G\amalg K$, we may apply
\cite[Theorem 2.9]{hz}, in order to deduce the existence of
$g_0\in G_0$ with $K_i^{gg_0}\cap K_{i'}^{g'g_0}\le G$. Now by \cite[Lemma 2.7]{HZ 07} 
two distinct maximal finite subgroups of
$G_0$ have trivial intersection. So we have $K_i^g\cap K_{i'}^{g'}=\ugp$, as needed.
\medskip

2) $v$ is the terminal vertex of $e_1$ and the initial vertex of
$e_2$. Then $\tilde G_{e_1}=K^g$ and $\tilde G_{e_2}=K_i^{g'}$ for
$g,g'\in G_0$ so they intersect trivially by the definition of
$G_0$ and \cite[Theorem 2.9]{hz}. So the Claim holds.

\bigskip

Now pick a torsion element $t\in \tilde G$ and $\tilde f\in\tilde F$ with
$t^{\tilde f}=t$. Let $e\in E(S)$ be an edge stabilized by $t$. Then $\tilde fe$
is also stabilized by $t$ and, as by \cite[Theorem 3.7]{RZ2},
the fixed set $S^t$ is a subtree, the path $[e,\tilde fe]$ is fixed by
$t$ as well. Note that $e$ and $\tilde{f}e$ cannot have a common vertex, since $\tilde{f}$ cannot stabilize any vertex.
Moreover, $S^t$ is {\em infinite} since $\gp{\tilde f}$ is torsion free and acts freely on $S^t$.
Now $S^t$ is connected and \prefcor{db} in the Appendix implies that it
must have path components of arbitrary cardinality. 
Therefore we can choose $e$ and $\tilde f$ such that $[e,\tilde fe]$ 
contains at least 3 pairwise adjacent edges and so $[e,\tilde fe]$
contains at least one vertex which 
is not the terminal point of all its incident edges.
Then by the Claim $t=1$.
\end{proof}

\appendix
\section*{Appendix: Path components of finite diameter in a profinite graph}
We shall need a general result about {\em profinite graphs}. 
Composition $RS$ of binary relations $R$ and $S$ on a set $X$ 
is defined as $x RS y$ if and only if there is $z\in X$
so that $(x,z)\in R$ and $(z,y)\in S$ holds. 
Define inductively $R^1:=R$ and $R^{n+1}:=R^nR$ for $n\in\N$. Let $R^o$ denote the 
{\em converse} relation, i.e., $(x,y)\in R^o$ if and only if $(y,x)\in R$ and, as common, $\Delta:=\{(x,x)\mid x\in X\}$
is the {\em diagonal}. 

For an abstract graph $\Gamma$ consider 
$R_0:=\{(x,y)\in\Gamma\times\Gamma\mid d_1(x)=d_0(y)\}$, and, set $R:=R_0\cup R_0^o\cup\Delta$. 
Then $x R^n y$ if and only if there is a geodesic of length not exceeding $n$ in $\Gamma$
containing  $x$ and $y$. The path-components of $\Gamma$ turn out to be the equivalence classes of $\Sigma:=\bigcup_nR^n$.
Define $\delta(\Gamma)$ to be the supremum of the diameters of its connected components then 
$$\delta(\Gamma)\le n \text{ if and only if } \Sigma=R^n.$$ 
When $X$ is a compact space then a standard compactness argument
implies that with $R$ compact every $R^n$ is compact.
Every profinite graph is an abstract one.

\begin{lemma}\label{l:geo-discrete}
Let $\Gamma$ be a profinite graph and $\delta(\Gamma)<\infty$. Then the path components of $\Gamma$
are exactly the connected components.
\end{lemma}

\begin{proof}
Set $n:=\delta(\Gamma)$. Since $R$ is closed $\Sigma=R^n$ is a closed equivalence relation.
Hence its equivalence classes, the path components, are all closed. The quotient graph
$\Gamma/\Sigma$ does not contain edges and so it is totally disconnected.
Since every connected component of $x\in\Gamma$ contains the path component of $x$,  each
path component is a connected component of $\Gamma$.
\end{proof}

\bc{db}
A connected profinite graph $\Gamma$ with $\delta(\Gamma)<\infty$ consists of a single path component.
\ec

\end{document}